\documentclass[10pt,oneside,leqno]{amsart}

\usepackage[cm]{fullpage}

\usepackage{amsmath,amsfonts,amssymb}
\usepackage{graphicx}
\usepackage[all]{xy}
\usepackage[english]{babel}
\usepackage[utf8x]{inputenc}
\usepackage{lscape}
\usepackage{url}
\usepackage{booktabs}
\usepackage{multirow}
\usepackage{array}
\usepackage{ifthen}
\usepackage{color}
\usepackage{paralist}

 \newtheorem{thm}{Theorem}[section]
 \newtheorem{prop}[thm]{Proposition}
 
 \newtheorem{lemma}[thm]{Lemma}
 
\theoremstyle{definition}
 \newtheorem{defi}[thm]{Definition}
 \newtheorem{rem}[thm]{Remark}

\newcommand{\N}{\mathbb{N}}

\newcommand{\R}{\mathbb{R}}
\newcommand{\C}{\mathbb{C}}
\newcommand{\st}{\;:\;}

\newcommand{\norma}[2]{\left\|#1\right\|_{\mathrm{L}^2_{#2}}}
\newcommand{\normazero}[1]{\left\|#1\right\|_{\mathrm{L}^2}}
\newcommand{\scalar}[3]{\left\langle  #1 \,\left|\, #2  \right. \right\rangle_{\mathrm{L}^2_{#3}}}
\newcommand{\scalard}[2]{\left\langle #1 \,\left|\, #2 \right. \right\rangle}

\newcommand{\sspace}{\cdot}
\newcommand{\ssspace}{\cdot\cdot}

\newcommand{\duale}[1]{#1^*}

\newcommand{\kth}[1]{\ifthenelse{\equal{#1}{1}}{$#1^\text{st}$}{\ifthenelse{\equal{#1}{2}}{$#1^\text{nd}$}{\ifthenelse{\equal{#1}{3}}{$#1^\text{rd}$}{$#1^\text{th}$}}}}

\newcommand{\ssum}[1]{\widetilde{\sum_{#1}}}

\newcommand{\Leb}[2]{\ifthenelse{\equal{#2}{0}}{\mathrm{L}^{2}\left(X;\wedge^{#1}T^*X\right)}{\mathrm{L}^{2}_{#2}\left(X;\wedge^{#1}T^*X\right)}}
\newcommand{\LebK}[2]{\ifthenelse{\equal{#2}{0}}{\mathrm{L}^{2}\left(K;\wedge^{#1}T^*X\right)}{\mathrm{L}^{2}_{#2}\left(K;\wedge^{#1}T^*X\right)}}
\newcommand{\Lebloc}[1]{\mathrm{L}^{2}_{\text{loc}}\left(X;\wedge^{#1}T^*X\right)}
\newcommand{\Cinf}[1]{\ifthenelse{\equal{#1}{0}}{\mathcal{C}^\infty\left(X;\R\right)}{\mathcal{C}^\infty\left(X;\wedge^{#1}T^*X\right)}}
\newcommand{\Cinfc}[1]{\ifthenelse{\equal{#1}{0}}{\mathcal{C}^\infty_{\mathrm{c}}\left(X;\R\right)}{\mathcal{C}^\infty_{\mathrm{c}}\left(X;\wedge^{#1}T^*X\right)}}
\newcommand{\Sob}[3]{\ifthenelse{\equal{#2}{0}}{\mathrm{W}^{#3, 2}\left(X;\wedge^{#1}T^*X\right)}{\mathrm{W}^{#3, 2}_{#2}\left(X;\wedge^{#1}T^*X\right)}}
\newcommand{\Sobloc}[2]{\mathrm{W}^{#1, 2}_{\textrm{loc}}\left(X;\wedge^{#2}T^*X\right)}
\newcommand{\SobK}[3]{\ifthenelse{\equal{#2}{0}}{\mathrm{W}^{#3, 2}\left(K;\wedge^{#1}T^*X\right)}{\mathrm{W}^{#3, 2}_{#2}\left(K;\wedge^{#1}T^*X\right)}}

\newcommand{\sign}[2]{
 \mathrm{sign}
 \left(
 \begin{array}{c}
   #1 \\
   #2
 \end{array}
 \right)
\,}

\newcommand{\der}[2]{\frac{\del #1}{\del x^{#2}}\,}

\newcommand{\expp}[1]{\exp\left(#1\right)\,}

\newcommand{\In}[1]{\left\{1,\ldots,#1\right\}}

\newcommand{\End}[1]{\mathrm{End}\left(#1\right)}

\newcommand{\Hom}[2]{\mathrm{Hom}\left(#1,\,#2\right)}

\newcommand{\sgn}[1]{\mathrm{sgn}\left(#1\right)}

\newcommand{\Sym}[1]{\mathrm{Sym}^{2}\left(#1\right)}

\newcommand{\arxiv}[1]{\texttt{#1}}

\newcommand{\paragrafo}[2]{\smallskip \paragraph{\texttt{Step {#1}} -- {\itshape #2}.\ }}

\newcommand{\destar}[2]{\duale{\de}_{#1,\,#2}}

\newcommand{\qednero}{\hfill$\blacksquare$}

\DeclareMathOperator{\dom}{dom}
\DeclareMathOperator{\imm}{im}
\DeclareMathOperator{\de}{d}

\DeclareMathOperator{\vol}{vol}
\DeclareMathOperator{\Hess}{Hess}
\DeclareMathOperator{\PSH}{PSH}
\DeclareMathOperator{\Levi}{L}
\DeclareMathOperator{\supp}{supp}

\newcommand{\del}{\partial}

\newcommand{\interno}[1]{\mathrm{int}\left({#1}\right)}

\allowdisplaybreaks[1]

\title{A vanishing result for strictly $p$-convex domains}
\author{Daniele Angella}
\address[Daniele Angella]{Dipartimento di Matematica\\
Universit\`{a} di Pisa \\
Largo Bruno Pontecorvo 5, 56127\\ 
Pisa, Italy}
\email{angella@mail.dm.unipi.it}

\author{Simone Calamai}
\address[Simone Calamai]{Scuola Normale Superiore\\
Piazza dei Cavalieri 7, 56126\\
Pisa, Italy}
\email{simone.calamai@sns.it}

\thanks{This work was supported by GNSAGA of INdAM}

\begin{document}

\vspace{-2cm}
\begin{minipage}[l]{10cm}
{\sffamily
  D. Angella, S. Calamai, A vanishing result for strictly $p$-convex domains,
  \textsc{doi:} \texttt{10.1007/s10231-012-0315-5}, to appear in {\em Ann. Mat. Pura Appl.}.

\smallskip

  \begin{flushright}\begin{footnotesize}
  (The original publication is available at \url{www.springerlink.com}.)
  \end{footnotesize}\end{flushright}
}
\end{minipage}
\vspace{2cm}

\begin{abstract}
In view of A. Andreotti and H. Grauert’s vanishing theorem for $q$-complete domains in $\C^n$, \cite{andreotti-grauert}, we re-prove a vanishing result by J.-P. Sha, \cite{sha}, and H. Wu, \cite{wu}, for the de Rham cohomology of strictly $p$-convex domains in $\R^n$ in the sense of F.~R. Harvey and H.~B. Lawson, \cite{harvey-lawson-2}.
Our proof uses the $\mathrm{L}^2$-techniques developed by L. H\"ormander, \cite{hormander}, and A. Andreotti and E. Vesentini, \cite{andreotti-vesentini}.
\end{abstract}

\maketitle

\section*{Introduction}

A weaker condition than holomorphic convexity for domains in $\C^n$ has been introduced by A. Andreotti and H. Grauert in \cite{andreotti-grauert}, defining \emph{$q$-complete domains} as domains in $\C^n$ admitting a proper exhaustion function whose Levi form has $n-p+1$ positive eigen-values.

In a recent series of foundational papers, \cite{harvey-lawson-1, harvey-lawson-2}, and references therein, F.~R. Harvey and H.~B. Lawson raise the interest on generalizations of the concept of convexity for Riemannian manifolds, proving many important results for \emph{$p$-convex} manifolds: namely, starting with a Riemannian manifold $(X,\, g)$, they ask whether it admits an exhaustion function whose Hessian is positive definite or satisfies weaker positive conditions.\\
Interpolating between the classical notions of convex functions and pluri-sub-harmonic functions, in \cite{harvey-lawson-2} they define the class of \emph{$p$-pluri-sub-harmonic functions} in terms of the positivity of the minors of their Hessian form, and they study \emph{$p$-convex domains}, which can be regarded as domains in $\R^n$ endowed with a smooth $p$-pluri-sub-harmonic proper exhaustion function.\\
The notions of \emph{geometric pluri-sub-harmonicity} and \emph{geometric convexity}, introduced and studied by  F.~R. Harvey and H.~B. Lawson in \cite{harvey-lawson-1}, is closely related to holomorphic convexity and $q$-completeness in the sense of A. Andreotti and H. Grauert, \cite{andreotti-grauert}.

In the complex case, holomorphic convexity and, more in general, $q$-completeness provide vanishing theorems for the Dolbeault cohomology (\cite{hormander}, respectively \cite{andreotti-grauert, andreotti-vesentini}).

We are concerned in studying vanishing results for strictly $p$-convex domains in $\R^n$ in the sense of F.~R. Harvey and H.~B. Lawson. More precisely, we give a proof of the following result.

\smallskip
\noindent {\bfseries Theorem \ref{thm:vanishing}.\ }
{\itshape
 Let $X$ be a strictly $p$-convex domain in $\R^n$. Then $H^k_{dR}(X;\R)=\{0\}$ for every $k\geq p$.
}
\smallskip

As pointed out to us by F.~R. Harvey and H.~B. Lawson, the above result was already known, as a consequence of \cite[Theorem 1]{sha} by J.-P. Sha,
and \cite[Theorem 1]{wu} by H. Wu, see also \cite[Proposition 5.7]{harvey-lawson-2}:
more precisely, they prove, using Morse theory, that the existence of a smooth proper strictly $p$-pluri-sub-harmonic exhaustion function
has consequences on the homotopy type of the domain.\\
In spite of this, our proof differs in the techniques, which are inspired by A. Andreotti and E. Vesentini \cite{andreotti-vesentini}:
in particular, the $\mathrm{L}^2$-techniques used in our proof could be hopefully applied in a wider context,
a fact which we would like to investigate further in future work.

\medskip

The organization of the paper is as follows. In Section \ref{sec:harvey-lawson}, we recall the main definitions introduced in \cite{harvey-lawson-1, harvey-lawson-2}, and the results proven by A. Andreotti and H. Grauert in \cite{andreotti-grauert}. In Section \ref{sec:preliminaries}, we prove some useful estimates, which will be used in Section \ref{sec:main-thm} to prove Theorem \ref{thm:vanishing}.

\medskip

\noindent{\sl Acknowledgments.}
The authors are warmly grateful to Adriano Tomassini and Xiuxiong Chen for their constant support and encouragement,
and to Giuseppe Tomassini for clear teaching on the H\"ormander $\mathrm{L}^2$-techniques.
The authors would like to thank Adriano also for his many useful suggestions.
Many thanks are due also to F.~R. Harvey and H.~B. Lawson, for bringing authors' attention to the results by J.-P. Sha, \cite{sha}, and H. Wu, \cite{wu}, and for further fruitful suggestions and remarks.
They are very grateful to the anonymous referee for his/her valuable comments, which highly improved the presentation of the paper.

\section{The notion of $p$-convexity by Harvey and Lawson}\label{sec:harvey-lawson}

Following F.~R. Harvey and H.~B. Lawson, \cite{harvey-lawson-2, harvey-lawson-1}, firstly we recall point-wise definitions of $p$-positive symmetric endomorphisms, then we will turn to manifolds, and finally we will recall the notion of $p$-pluri-sub-harmonic (exhaustion) functions and (strictly) $p$-convex domains.

\subsection{$p$-positive (sections of) symmetric endomorphisms}

Let $\left(V,\, \scalard{\sspace}{\ssspace}\right)$ be an $n$-dimensional real inner product space. Let $G:V\to \duale{V}$ denote the isomorphism defined as $G(v):= \scalard{v}{\sspace}$.\\
Let $\Sym{V}$ denote the space of symmetric elements of $\duale{\left(V \otimes V\right)}$; namely, $A\in \Sym{V}$ if and only if $A(v \otimes w) = A(w \otimes v)$, for any $v,w\in V$. By means of the inner product $\scalard{\sspace}{\ssspace}$, the space $\Sym{V}$ is isomorphic to the space of the $\scalard{\sspace}{\ssspace}$-symmetric endomorphisms of $V$: given $A \in \Sym{V}$, we denote by $G^{-1}A \in \Hom{V}{V}$ the corresponding $\scalard{\sspace}{\ssspace}$-symmetric endomorphism.

The endomorphism $G^{-1}A\in\Hom{V}{V}$ extends to $D_{G^{-1}A}^{[p]}\in\Hom{\wedge^p V}{\wedge^p V}$; namely, on a simple vector $v_{i_1} \wedge \cdots \wedge v_{i_p}\in\wedge^pV$, the endomorphism $D_{G^{-1}A}^{[p]}$ acts as
\begin{align*}
 D_{G^{-1}A}^{[p]} \left( v_{i_1} \wedge \cdots \wedge v_{i_p} \right) \;:=\; 
 \sum_{\ell=1}^{p} v_{i_1}\wedge\cdots\wedge v_{i_{\ell-1}}\wedge G^{-1}A \left(v_{i_\ell}\right) \wedge v_{i_{\ell+1}}\wedge\cdots\wedge v_{i_p} \;.
\end{align*}
Observe that $D_{G^{-1}A}^{[p]}\in\Hom{\wedge^p V}{\wedge^p V}$ is a symmetric endomorphism with respect to the scalar product on $\wedge^pV$ induced by $\scalard{\sspace}{\ssspace}$.

Finally, given a $\scalard{\sspace}{\ssspace}$-symmetric endomorphism  $E\in \Hom{V}{V}$, let $\sgn{E}$ denote the number of non-negative eigenvalues of $E$. 

Notice that, given  $A\in \Sym{V}$, and given two inner products  on $V$ inducing respectively the isomorphisms $G_1$ and $G_2$, then there holds $\sgn{G_1^{-1}A} = \sgn{G_2^{-1}A}$; it is also important to notice that, for $p>1$, it might hold $\sgn{D_{G_1^{-1}A}^{[p]}} \neq  \sgn{D_{G_2^{-1}A}^{[p]}}$, since the eigenvalues of $D_{G^{-1}A}^{[p]}$ are of the form
\[
 \lambda_{i_1} + \cdots + \lambda_{i_p} \qquad \text{ for } i_1, \ldots, i_p \in \In{n}  \text{ s.t. } i_1 < \cdots < i_p \;, 
\]
where $\lambda_1, \, \ldots , \, \lambda_n$ are the eigenvalues of $G^{-1}A$.

\begin{defi}[{\cite{harvey-lawson-2, harvey-lawson-1}}]\label{defi:p-positivity-on-vectors}
\begin{itemize}
 \item Let $V$ be a $\R$-vector space endowed with an inner product $\scalard{\sspace}{\ssspace}$.
 Denote the space of \emph{$p$-positive forms of \kth{k} branch} on $V$ as
\begin{align*}
 \mathcal{P}_p^{(k)}\left(V,\, \scalard{\sspace}{\ssspace}\right):= 
\left\{ 
A \in \Sym{V} \st \sgn{ D_{G^{-1}A}^{[p]}} \geq \binom{n}{p} -k +1  
\right\}\;.
\end{align*}

 \item Let $\left(X ,\, g \right)$ be a Riemannian manifold.
Define the space of \emph{$p$-positive sections of \kth{k} branch} of the bundle $\Sym{TX}$ of symmetric endomorphisms of $TX$ as
\begin{align*}
 \mathcal{P}_p^{(k)}\left(X,\, g \right) \;:=\;
\left\{
 A \in \Sym{T X} \st \forall x\in X, \; A_x\in\mathcal{P}_p^{(k)}\left(T_xX, g_x\right) 
\right\}.
\end{align*}

\qednero
\end{itemize}
\end{defi}


\subsection{$p$-pluri-sub-harmonic functions}
In order to introduce an exhaustion of a given Riemannian manifold, we focus on special 
$p$-positive symmetric $2$-forms, those arising from the Hessian of smooth functions.

\medskip

Thus, let $\left(X,\, g\right)$ be a Riemannian manifold and let $u$ be a smooth real valued function on $X$.
Let $\nabla$ denote the Levi-Civita connection of the Riemannian metric $g$, and let
\begin{align*}
 \Hess u \left(V, W\right) \;:=\; V\, W\, u - \left(\nabla_V W\right)\, u \;,  
\end{align*}
where $V$ and $W$ are smooth sections of the tangent bundle $TX$.
Thus, $\Hess u(x)  \in \Sym{T_x X}$, for any $x \in X$.

\begin{defi}[{\cite{harvey-lawson-1}}]
Let $\left(X,\, g\right)$ be a Riemannian manifold.
\begin{itemize}
 \item The space
\begin{align*}
\PSH^{(k)}_p\left(X,\, g\right) \;:=\;
\left\{ u\in \Cinf{0} \st \Hess u \in \mathcal{P}^{(k)}_p \left( X,\, g\right) \right\} \;,
\end{align*}
is called the space of \emph{$p$-pluri-sub-harmonic functions of \kth{k} branch} on $X$.

\item The space
\begin{align*}
\interno{\PSH^{(k)}_p\left(X,\, g\right)} \;:=\;
\left\{ u\in \Cinf{0} \st \Hess u \in \interno{\mathcal{P}^{(k)}_p \left( X,\, g\right)} \right\} \;,
\end{align*}
(where $\interno{\mathcal{P}^{(k)}_p \left( X,\, g\right)}$ denotes the interior of $\mathcal{P}^{(k)}_p \left( X,\, g\right)$) is called the space of \emph{strictly $p$-pluri-sub-harmonic functions of \kth{k} branch} on $X$.
\qednero
\end{itemize}
\end{defi}

\subsection{(Strictly) $p$-convexity}

We are now ready to recall the concept of $p$-convexity, which is central in \cite{harvey-lawson-2}.

\medskip

Let $\left(X,\, g\right)$ be a Riemannian manifold.
Let $K\subseteq X$ be a compact set. The \emph{$p$-convex hull} of $K$ is given by
\begin{align*}
{\widetilde{K}}^{\PSH^{(1)}_p \left( X ,\, g\right)} \;:=\;
\left\{ x\in X \st
\forall \phi \in PSH^{(1)}_p \left( X ,\, g\right) ,\;
\phi(x) \leq \max_{y\in K}   \phi (y)
\right\} \;.
\end{align*}

\begin{defi}[{\cite{harvey-lawson-1}}]\label{defi:p-convexity}
Let $\left(X,\, g\right)$ be a Riemannian manifold.
Then, $X$ is called \emph{$p$-convex} if, for any compact set $K\subseteq X$, then ${\widetilde{K}}^{\PSH^{(1)}_p \left( X ,\, g\right)}$ is relatively compact in $X$.
\qednero
\end{defi}

\medskip

Define the \emph{$p$-core} of $X$, \cite[Definition 4.1]{harvey-lawson-1}, as
$$ \mathrm{Core}_p \left( X ,\, g \right) \;:=\; \left\{ x\in X \st \text{for all }u\in\PSH^{(1)}_p\left(X,\, g\right),\; \Hess u(x)\not\in\interno{\mathcal{P}^{(1)}_p\left(T_xX,\, g_x\right)} \right\} \;. $$

\begin{defi}[{\cite{harvey-lawson-1}}]
 Let $\left(X,\, g\right)$ be a Riemannian manifold. Then, $X$ is called \emph{strictly $p$-convex} if
\begin{inparaenum}[(\itshape i\upshape)]
 \item $\mathrm{Core}_p \left( X ,\, g \right) = \varnothing$ and,
 \item for any compact set $K\subseteq X$, then ${\widetilde{K}}^{\PSH^{(1)}_p\left( X ,\, g\right)}$ is relatively compact in $X$.
\end{inparaenum}
\qednero
\end{defi}

\subsection{(Strictly) $p$-convexity and (strictly) $p$-pluri-sub-harmonic exhaustion functions}
The following correspondences come from \cite{harvey-lawson-1}.
\begin{thm}[{\cite[Theorem 4.4, Theorem 4.8]{harvey-lawson-1}}]\label{thm:p-positive-exhaustion-functions}
Let $\left(X,\, g\right)$ be a Riemannian manifold.
Then $X$ is $p$-convex (respectively, strictly $p$-convex) if and only if $X$ admits a smooth proper exhaustion function $u\in \PSH^{(1)}_p \left( X ,\, g\right)$ (respectively, $u\in \interno{\PSH^{(1)}_p \left( X ,\, g\right)}$).
\end{thm}

\subsection{The $p$-convexity and the $q$-completeness}

All along the definitions of the previous section, the special case that we had in mind is the following classical construction in Complex Analysis.

\medskip

In \cite{andreotti-grauert}, A. Andreotti and H. Grauert pointed out the following concept.
\begin{defi}[{\cite{andreotti-grauert}}]\label{defi:p_pseudoconvexity_original}
Let $D \subseteq \mathbb{C}^n$ be a domain, and let $\phi$ be a smooth real-valued function on $D$. The function $\phi$ is called \emph{$p$-pluri-sub-harmonic} (respectively, \emph{strictly $p$-pluri-sub-harmonic}) if and only if, for any $z\in D$, the Hermitian form defined, for $\xi:=:\left(\xi^a\right)_{a\in\In{n}} \in \mathbb{C}^n$, as
\begin{align*}
 \Levi(\phi)_z\,(\xi) \;:=\; \sum_{a , b =1}^n 
\frac{\partial^2 \phi }{\partial z^a \partial \bar z^{b}}(z)\, \xi^a\, \overline{\xi^{b}} \;,
\end{align*}
has $n-p+1$ non-negative (respectively, positive) eigenvalues.
\qednero
\end{defi}


A. Andreotti and H. Grauert, in \cite{andreotti-grauert}, studied domains of $\mathbb{C}^n$ admitting strictly $q$-pluri-sub-harmonic exhaustion functions (the so-called \emph{$q$-complete domains}), proving a vanishing theorem for the higher-degree Dolbeault cohomology groups of such domains; then A. Andreotti and E. Vesentini, in \cite{andreotti-vesentini}, re-prove the same result extending the $\mathrm{L}^2$-techniques by L. H\"ormander, \cite{hormander}.

\bigskip

Thus, in the same vein as A. Andreotti and H. Grauert, we would consider domains $X$ in $\R^n$ endowed with an exhaustion function $u\in\mathcal{C}^\infty(X;\R)$ whose Hessian is in $\interno{\mathcal{P}_p^{(1)}(X,g)}$, proving a vanishing result for the higher-degree de Rham cohomology groups for strictly $p$-convex domains in the sense of F.~R. Harvey and H.~B. Lawson.

\section{Vanishing of the de Rham cohomology for strictly $p$-convex domains}\label{sec:preliminaries}

Let $X$ be an oriented Riemannian manifold of dimension $n$, and denote by $g$ its Riemannian metric and by $\vol$ its volume. The Riemannian metric $g$ induces, for every $x\in X$, a point-wise scalar product $\left\langle \sspace \left| \ssspace \right. \right\rangle_{g_x}\colon \wedge^\bullet T^*_xX \times \wedge^\bullet T^*_xX \to \R$.

Fix $\phi\in\mathcal{C}^0(X;\R)$ a continuous function. For every $\varphi,\, \psi \in \Cinfc{\bullet}$, let
$$ \scalar{\varphi}{\psi}{\phi} \;:=\; \int_X \left\langle \varphi \left| \psi \right. \right\rangle_{g_x} \, \expp{-\phi} \vol \;\in\; \R \;, $$
and, for $k\in\N$, define $\Leb{k}{\phi}$ as the completion of the space $\Cinfc{k}$ of smooth $k$-forms with compact support, with respect to the metric induced by $\norma{\sspace}{\phi}:=\scalar{\sspace}{\sspace}{\phi}$.
Therefore, the space $\Leb{k}{\phi}$ is a Hilbert space, endowed with the scalar product $\scalar{\sspace}{\ssspace}{\phi}$, and $\Cinfc{k}$ is dense in $\Leb{k}{\phi}$. For any $k\in\N$, let $\Lebloc{k}$ denote the space of $k$-forms $f$ whose restriction $f\lfloor_{K}$ to every compact set $K\subseteq X$ belongs to $\LebK{k}{}$.

\medskip

For every $\phi_1,\,\phi_2\in\mathcal{C}^0(X;\R)$, the operator
$$ \de\colon \Leb{\bullet}{\phi_1} \dashrightarrow \Leb{\bullet+1}{\phi_2} $$
is densely-defined and closed; denote by
$$ \destar{\phi_2}{\phi_1}\colon \Leb{\bullet+1}{\phi_2} \dashrightarrow \Leb{\bullet}{\phi_1} $$
its adjoint, which is a densely-defined closed operator.

\medskip

Recall that, on a domain $X$ in $\R^n$, fixed $k\in\N$, $s\in\N$, and $\phi\in\Cinf{0}$, the Sobolev space $\Sob{k}{\phi}{s}$ is the space of $k$-forms $f:=:\ssum{|I|=k}f_I\,\de x^I$ such that $\frac{\del^{\ell_1+\cdots+\ell_n} f_I}{\del^{\ell_1} x^1\cdots \del^{\ell_n}x^n}\in\Leb{k}{\phi}$ for every multi-index $\left(\ell_1,\ldots,\ell_n\right)\in\N^n$ such that $\ell_1+\cdots+\ell_n\leq s$, and for every strictly increasing multi-index $I$ such that $|I|=k$. The space $\Sobloc{s}{k}$ is defined as the space of $k$-forms $f$ whose restriction $f\lfloor_K$ to every compact set $K\subseteq X$ belongs to $\SobK{k}{}{s}$.

\medskip

As a matter of notation, the symbol $\ssum{|I|=k}$ denotes the sum over the strictly increasing multi-indices $I:=:\left(i_1,\ldots,i_k\right)\in\N^k$ (that is, the multi-indices such that $0<i_1<\cdots<i_k$) of length $k$. Given $I_1$ and $I_2$ two multi-indices of length $k$, let $\sign{I_1}{I_2}$ be the sign of the permutation $\left(\begin{array}{c}I_1\\I_2\end{array}\right)$ if $I_1$ is a permutation of $I_2$, and zero otherwise.

\subsection{Some preliminary computations}

Let $X$ be a domain in $\R^n$, that is, an open connected subset of $\R^n$, endowed with the metric and the volume induced, respectively, by the Euclidean metric and the standard volume of $\R^n$.

\medskip

For $\phi_1,\, \phi_2\in\Cinf{0}$, consider $\de\colon \Leb{k-1}{\phi_1} \dashrightarrow \Leb{k}{\phi_2}$.
The following lemma gives an explicit expression of the adjoint $\destar{\phi_2}{\phi_1}\colon \Leb{k}{\phi_2} \dashrightarrow \Leb{k-1}{\phi_1}$ (compare, e.g., with \cite[\S8.2.1]{dellasala-saracco-simioniuc-tomassini}, \cite[Lemma O.2]{gunning-1} in the complex case).

\begin{lemma} \label{lemma:d*}
 Let $X$ be a domain in $\R^n$. Let $\phi_1,\,\phi_2\in\Cinf{0}$ and consider
 $$
 \xymatrix{
 \Leb{k-1}{\phi_1} \ar@{-->}@/^1pc/[r]^{\de} & \Leb{k}{\phi_2} \ar@{-->}@/^1pc/[l]^{\destar{\phi_2}{\phi_1}}
 } \;.
 $$
 Let
 $$v \;:=:\; \ssum{|I|=k} v_I\, \de x^I \in \Leb{k}{\phi_2} $$
 and suppose that $v\in\dom \destar{\phi_2}{\phi_1}$. Then
 \begin{eqnarray*}
 \destar{\phi_2}{\phi_1} v &=& \expp{\phi_1} \destar{0}{0}\left(\expp{-\phi_2} v\right) \\[5pt]
 &=& \ssum{|J|=k-1} \left(-\expp{\phi_1} \ssum{|I|=k}\sum_{\ell=1}^{n} \sign{\ell J}{I}\, \der{\left(v_I\expp{-\phi_2}\right)}{\ell} \right) \, \de x^J \;.
 \end{eqnarray*}
\end{lemma}

\begin{proof}
By definition of $\destar{\phi_2}{\phi_1}$, for every $u\in\dom \de$, one has $\scalar{\de u}{v}{\phi_2} = \scalar{u}{\destar{\phi_2}{\phi_1} v}{\phi_1}$.
Hence, consider
$$ u \;:=:\; \ssum{|J|=k-1} u_J\, \de x^J \;\in\; \Cinfc{k-1} \;, $$
and compute
$$ \de u \;=\; \ssum{\substack{|J|=k-1\\ |I|=k}}\sum_{\ell=1}^{n} \sign{\ell J}{I} \der{u_J}{\ell} \de x^I \;. $$
The statement follows by computing
\begin{eqnarray*}
\scalar{\de u}{v}{\phi_2} &=& \int_X \ssum{\substack{|J|=k-1\\|I|=k}} \sum_{\ell=1}^{n} \sign{\ell J}{I} \der{u_J}{\ell} v_I \, \expp{-\phi_2}\, \vol \\[5pt]
 &=& -\int_X \ssum{\substack{|J|=k-1\\|I|=k}} \sum_{\ell=1}^{n} \sign{\ell J}{I} \der{\left(v_I \, \expp{-\phi_2}\right)}{\ell} u_J\, \vol
\end{eqnarray*}
and
$$ \scalar{u}{\destar{\phi_2}{\phi_1} v}{\phi_1} \;=\; \int_X \ssum{|J|=k-1} \left(\destar{\phi_2}{\phi_1} v\right)_J\, u_J\, \expp{-\phi_1}\, \vol \;, $$
where $\destar{\phi_2}{\phi_1} v=:\ssum{|J|=k-1} \left(\destar{\phi_2}{\phi_1} v\right)_J\,\de x^J$.
\end{proof}

\medskip

For any fixed $\phi\in\Cinf{0}$ and for any $j\in\In{n}$, define the operator
$$ \delta^\phi_j \colon \Cinf{0} \to \Cinf{0} \;, $$
where
$$ \delta^\phi_j(f) \;:=\; -\expp{\phi}\,\der{\left(f\,\expp{-\phi}\right)}{j} \;=\; \der{\phi}{j}\cdot f-\der{f}{j} \;. $$
The following lemma states that $\delta^\phi_j$ is the adjoint of $\der{}{j}$ in $\Leb{0}{\phi}$, and computes the commutator between $\delta^\phi_j$ and $\der{}{k}$ (compare with, e.g., \cite[pages 83-84]{hormander}).

\begin{lemma}\label{lemma:delta}
 Let $X$ be a domain in $\R^n$. Let $\phi\in\Cinf{0}$ and $j\in\In{n}$, and consider the operator $\delta^\phi_j\colon \Cinf{0} \to \Cinf{0}$. Then:
\begin{itemize}
 \item for every $w_1,w_2\in\Cinfc{0}$,
$$ \int_X w_1\cdot\der{w_2}{k}\expp{-\phi}\vol \;=\; \int_X\delta^\phi_k(w_1)\cdot w_2\,\expp{-\phi}\vol \;;$$
 \item for any $k\in\In{n}$, the following commutation formula holds in $\End{\Cinfc{0}}$:
$$ \left[\delta^\phi_j,\, \der{}{k}\right] \;=\; -\frac{\del^2\phi}{\del x^j\del x^k}\cdot \;.$$
\end{itemize}
\end{lemma}

\medskip

Finally, we prove the following estimate, which will be used in the proof of Theorem \ref{thm:vanishing} (we refer to \cite[\S4.2]{hormander}, or, e.g., \cite[Lemma O.3]{gunning-1} and \cite[\S8.3.1]{dellasala-saracco-simioniuc-tomassini} for its complex counterpart).

\begin{prop}\label{prop:stima}
 Let $X$ be a domain in $\R^n$ and $\phi,\,\psi\in\Cinf{0}$. Consider
 $$
 \xymatrix{
 \Leb{k-1}{\phi-2\psi} \ar@{-->}@/^1pc/[r]^{\de} & \Leb{k}{\phi-\psi} \ar@{-->}@/^1pc/[r]^{\de} \ar@{-->}@/^1pc/[l]^{\destar{\phi-\psi}{\phi-2\psi}} & \Leb{k+1}{\phi} \ar@{-->}@/^1pc/[l]^{\destar{\phi}{\phi-\psi}} \;.
 }
 $$
 Then, for any $\eta :=: \ssum{|I|=k}\eta_I \, \de x^I \in \Cinfc{k}$, one has
\begin{eqnarray*}
 \lefteqn{\int_X \ssum{\substack{|J|=k-1\\|I_1|=k\\|I_2|=k}}\sum_{\ell_1,\,\ell_2=1}^{n} \sign{\ell_1 J}{I_1}\sign{\ell_2 J}{I_2} \frac{\del^2\phi}{\del x^{\ell_1}\,\del x^{\ell_2}}\, \eta_{I_1}\, \eta_{I_2} \, \expp{-\phi} \vol} \\[5pt]
 && \leq \; C \cdot \left( \norma{\destar{\phi-\psi}{\phi-2\psi} \eta}{\phi-2\psi}^2 + \norma{\de\eta}{\phi}^2 + \int_X \ssum{|I|=k}\sum_{\ell=1}^{n}\left|\frac{\del\psi}{\del x^\ell}\right|^2\,\left|\eta_I\right|^2\, \expp{-\phi} \vol \right) \;,
\end{eqnarray*}
where $C:=:C(k,n)\in\N$ is a constant depending just on $k$ and $n$.
\end{prop}

\begin{proof}
It is straightforward to compute
$$ \de \eta \;=\; \ssum{\substack{|I|=k\\|H|=k+1}}\sum_{\ell=1}^{n} \sign{\ell I}{H}\der{\eta_I}{\ell}\de x^H $$
and, using Lemma \ref{lemma:d*},
\begin{eqnarray*}
 \destar{\phi-\psi}{\phi-2\psi} \eta &=& -\expp{-\psi} \ssum{\substack{|J|=k-1\\|I|=k}}\sum_{\ell=1}^{n}\sign{\ell J}{I} \left(\der{\eta_I}{\ell}-\der{\left(\phi-\psi\right)}{\ell}\eta_I\right) \, \de x^J \\[5pt]
 &=& \expp{-\psi} \ssum{\substack{|J|=k-1\\|I|=k}}\sum_{\ell=1}^{n}\sign{\ell J}{I} \left(\delta^\phi_\ell\left(\eta_I\right)-\der{\psi}{\ell}\,\eta_I\right) \, \de x^J \;.
\end{eqnarray*}
For every $J$ such that $|J|=k-1$, the previous equality gives
$$
\ssum{|I|=k}\sum_{\ell=1}^{n}\sign{\ell J}{I} \delta^\phi_\ell\left(\eta_I\right)\;=\; \expp{\psi} \left(\destar{\phi-\psi}{\phi-2\psi} \eta\right)_J  + \ssum{|I|=k}\sum_{\ell=1}^{n}\sign{\ell J}{I} \der{\psi}{\ell}\,\eta_I \;,
$$
where $\destar{\phi-\psi}{\phi-2\psi}\eta =: \ssum{|J|=k-1}\left(\destar{\phi-\psi}{\phi-2\psi}\eta\right)_J\de x^J$.\\
By the arithmetic mean--geometric mean inequality, one gets
\begin{eqnarray}
 \lefteqn{\int_X \ssum{|J|=k-1} \left| \ssum{|I|=k}\sum_{\ell=1}^{n}\sign{\ell J}{I} \delta^\phi_\ell\left(\eta_I\right) \right|^2 \, \expp{-\phi} \vol} \label{eq:stima-delta}\\[5pt]
 && \leq\; 2\, \int_X \ssum{|J|=k-1}\left(\left|\left(\destar{\phi-\psi}{\phi-2\psi}\eta\right)_J\right|^2\, \expp{2\psi} + \left|\ssum{|I|=k}\sum_{\ell=1}^n\sign{\ell J}{I} \der{\psi}{\ell}\eta_I\right|^2 \right)\, \expp{-\phi} \vol \nonumber\\[5pt]
 && \leq\; C \, \left(\norma{\destar{\phi-\psi}{\phi-2\psi}\eta}{\phi-2\psi}^2+\int_X\ssum{|I|=k}\sum_{\ell=1}^{n}\left|\der{\psi}{\ell}\right|^2\cdot\left|\eta_I\right|^2\, \expp{-\phi} \vol\right) \;,\nonumber
\end{eqnarray}
where $C:=:C(k,n)\in\N$ depends on $k$ and $n$ only.\\
Now, using Lemma \ref{lemma:delta}, one computes
\begin{eqnarray}
 \lefteqn{\int_X \ssum{|J|=k-1} \left| \ssum{|I|=k}\sum_{\ell=1}^{n} \sign{\ell J}{I} \delta^\phi_\ell\left(\eta_I\right) \right|^2 \, \expp{-\phi} \vol} \label{eq:normale-delta}\\[5pt]
 && =\; \ssum{|J|=k-1} \ssum{\substack{|I_1|=k\\|I_2|=k}}\sum_{\ell_1,\,\ell_2=1}^{n} \sign{\ell_1 J}{I_1} \sign{\ell_2 J}{I_2} \int_X \delta^\phi_{\ell_1}\left(\eta_{I_1}\right) \cdot \delta^\phi_{\ell_2}\left(\eta_{I_2}\right) \expp{-\phi} \vol \nonumber\\[5pt]
 && =\; \ssum{\substack{|J|=k-1\\|I_1|=k\\|I_2|=k}}\sum_{\ell_1,\,\ell_2=1}^{n} \sign{\ell_1 J}{I_1} \sign{\ell_2 J}{I_2} \int_X \left(\der{\eta_{I_1}}{\ell_2}\,\der{\eta_{I_2}}{\ell_1}+\frac{\del^2\phi}{\del x^{\ell_1}\,\del x^{\ell_2}}\,\eta_{I_1}\,\eta_{I_2}\right) \, \expp{-\phi} \vol \;.\nonumber
\end{eqnarray}
Now, note that
\begin{eqnarray}
 \left|\de\eta\right|^2 &=& \ssum{|H|=k+1} \left|\ssum{|I|=k}\sum_{\ell=1}^{n}\sign{\ell I}{H}\der{\eta_I}{\ell}\right|^2 \label{eq:norma-de}\\[5pt]
 &=& \ssum{|H|=k+1} \left(\ssum{\substack{|I_1|=k\\|I_2|=k}}\sum_{\ell_1,\,\ell_2=1}^{n}\sign{\ell_1 I_1}{H}\sign{\ell_2 I_2}{H}\der{\eta_{I_1}}{\ell_1}\der{\eta_{I_2}}{\ell_2}\right) \nonumber\\[5pt]
 &=& \ssum{\substack{|I_1|=k\\|I_2|=k}}\sum_{\ell_1,\,\ell_2=1}^{n}\sign{\ell_1 I_1}{\ell_2 I_2}\der{\eta_{I_1}}{\ell_1}\der{\eta_{I_2}}{\ell_2} \nonumber\\[5pt]
 &=& \ssum{|I|=k}\sum_{\ell=1}^{n} \left|\der{\eta_I}{\ell}\right|^2 - \ssum{\substack{|J|=k-1\\|I_1|=k\\|I_2|=k}} \sum_{\ell_1,\,\ell_2=1}^{n} \sign{\ell_1 J}{I_1} \sign{\ell_2 J}{I_2} \der{\eta_{I_1}}{\ell_2} \der{\eta_{I_2}}{\ell_1} \;.\nonumber
\end{eqnarray}
Hence, in view of \eqref{eq:norma-de}, \eqref{eq:normale-delta}, \eqref{eq:stima-delta}, we get
\begin{eqnarray*}
 \lefteqn{\int_X \ssum{\substack{|J|=k-1\\|I_1|=k\\|I_2|=k}}\sum_{\ell_1,\,\ell_2=1}^{n} \sign{\ell_1 J}{I_1}\sign{\ell_2 J}{I_2} \frac{\del^2\phi}{\del x^{\ell_1}\,\del x^{\ell_2}}\, \eta_{I_1}\, \eta_{I_2} \, \expp{-\phi} \vol} \\[5pt]
 && \leq\; \int_X \left(\ssum{\substack{|J|=k-1\\|I_1|=k\\|I_2|=k}}\sum_{\ell_1,\,\ell_2=1}^{n} \sign{\ell_1 J}{I_1}\sign{\ell_2 J}{I_2} \frac{\del^2\phi}{\del x^{\ell_1}\,\del x^{\ell_2}}\, \eta_{I_1}\, \eta_{I_2}+\ssum{|I|=k}\sum_{\ell=1}^{n}\left|\frac{\del \eta_I}{\del x^\ell}\right|^2\right) \, \expp{-\phi} \vol \\[5pt]
 && =\; \int_X \left(\ssum{|J|=k-1} \left| \ssum{|I|=k}\sum_{\ell=1}^{n}\sign{\ell J}{I} \delta^\phi_\ell\left(\eta_I\right) \right|^2 + \ssum{|H|=k+1}\left|\left(\de \eta\right)_H\right|^2\right) \, \expp{-\phi} \vol \\[5pt]
 && \leq \; C \cdot \left( \norma{\destar{\phi-\psi}{\phi-2\psi}\eta}{\phi-2\psi}^2 + \norma{\de\eta}{\phi}^2 + \int_X \ssum{|I|=k}\sum_{\ell=1}^{n}\left|\frac{\del\psi}{\del x^\ell}\right|^2\,\left|\eta_I\right|^2\, \expp{-\phi} \vol \right) \;,
\end{eqnarray*}
concluding the proof.
\end{proof}

\begin{rem}\label{rem:stima}
 The argument in the proof of Proposition \ref{prop:stima} actually proves the following stronger estimate, which will be used in the regularization process in Theorem \ref{thm:vanishing}.\\
 {\itshape
 Let $X$ be a domain in $\R^n$ and $\phi,\,\psi\in\Cinf{0}$. Consider
 $$
 \xymatrix{
 \Leb{k-1}{\phi-2\psi} \ar@{-->}@/^1pc/[r]^{\de} & \Leb{k}{\phi-\psi} \ar@{-->}@/^1pc/[r]^{\de} \ar@{-->}@/^1pc/[l]^{\destar{\phi-\psi}{\phi-2\psi}} & \Leb{k+1}{\phi} \ar@{-->}@/^1pc/[l]^{\destar{\phi-\psi}{\phi-2\psi}} \;.
 }
 $$
 Then, for any $\eta :=: \ssum{|I|=k}\eta_I \, \de x^I \in \Cinfc{k}$, one has
 \begin{eqnarray*}
 \lefteqn{\int_X \left(\ssum{\substack{|J|=k-1\\|I_1|=k\\|I_2|=k}}\sum_{\ell_1,\,\ell_2=1}^{n} \sign{\ell_1 J}{I_1}\sign{\ell_2 J}{I_2} \frac{\del^2\phi}{\del x^{\ell_1}\,\del x^{\ell_2}}\, \eta_{I_1}\, \eta_{I_2}+\ssum{|I|=k}\sum_{\ell=1}^{n}\left|\frac{\del \eta_I}{\del x^\ell}\right|^2\right) \, \expp{-\phi} \vol} \\[5pt]
 && \leq \; C \cdot \left( \norma{\destar{\phi-\psi}{\phi-2\psi}\eta}{\phi-2\psi}^2 + \norma{\de\eta}{\phi}^2 + \int_X \ssum{|I|=k}\sum_{\ell=1}^{n}\left|\frac{\del\psi}{\del x^\ell}\right|^2\,\left|\eta_I\right|^2\, \expp{-\phi} \vol \right) \;,
 \end{eqnarray*}
 where $C:=:C(k,n)\in\N$ is a constant depending just on $k$ and $n$.
 }
\end{rem}

\section{Proof of the main theorem}\label{sec:main-thm}

We are ready to prove the following vanishing theorem for the higher-degree de Rham cohomology groups of a strictly $p$-convex domain in $\R^n$
(for a different proof, involving Morse theory, compare \cite[Theorem 1]{sha} by J.-P. Sha, and \cite[Theorem 1]{wu} by H. Wu,
see also \cite[Proposition 5.7]{harvey-lawson-2}).

\begin{thm}\label{thm:vanishing}
 Let $X$ be a strictly $p$-convex domain in $\R^n$. Then $H^k_{dR}(X;\R)=\{0\}$ for every $k\geq p$.
\end{thm}

\begin{proof}
We are going to prove that every $\de$-closed $k$-form $\eta\in\Cinf{k}$ is $\de$-exact, namely, there exists $\alpha\in\Cinf{k-1}$ such that $\eta=\de\alpha$; the statement of the theorem is a direct consequence of this result.
Let us split the proof in the following steps.

\paragrafo{1}{Definitions of the weight functions and other notations}
Being $X$ a strictly $p$-convex domain in $\R^n$, by F.~R. Harvey and H.~B. Lawson's \cite[Theorem 4.8]{harvey-lawson-1} (see also \cite[Theorem 5.4]{harvey-lawson-2}), there exists a smooth proper strictly $p$-pluri-sub-harmonic exhaustion function
$$ \rho \;\in\; \interno{\PSH^{(1)}_p\left( X,\, g\right)} \cap \Cinf{0} \;,$$
where $g$ is the metric on $X$ induced by the Euclidean metric on $\R^n$.\\
For every $m\in\N$, consider the compact set
$$ K^{(m)} \;:=\; \left\{x\in X \st \rho(x)\leq m \right\} \;,$$
and define
$$ L^{(m)} \;:=\; \min_{K^{(m)}} \lambda_{1}^{[k]} \;>\; 0 \;,$$
where, for every $x\in X$, the real numbers $\lambda_{1}^{[k]}(x)\leq\cdots\leq \lambda_{\binom{n}{k}}^{[k]}(x)$ are the ordered eigen-values of $D^{[k]}_{g^{-1}\Hess \rho(x)}\in\Hom{\wedge^kT_xX}{\wedge^kT_xX}$, and $\lambda_1(x)\leq\cdots\leq\lambda_n(x)$ are the ordered eigen-values of $g^{-1}\Hess \rho(x)\in\Hom{T_xX}{T_xX}$; indeed, note that, for every $x\in X$,
$$ \lambda_{1}^{[k]}(x) \;=\; \lambda_1(x)+\cdots+\lambda_{k}(x) \;\geq\; \lambda_1(x)+\cdots+\lambda_p(x) \;>\; 0 \;,$$
being $\rho$ strictly $p$-pluri-sub-harmonic, and that the function $X\ni x\mapsto \lambda_{1}^{[k]}(x)\in\R$ is continuous.\\
Fix $\left\{\rho_\nu\right\}_{\nu\in\N}\subset \Cinfc{0}$ such that
\begin{inparaenum}[(\itshape i\upshape)]
 \item $0\leq \rho_\nu \leq 1$ for every $\nu\in\N$, and
 \item for every compact set $K\subseteq X$, there exists $\nu_0:=:\nu_0(K)\in\N$ such that $\rho_\nu\lfloor_{K}=1$ for every $\nu\geq \nu_0$.
\end{inparaenum}\\
Then, we can choose $\psi\in\Cinf{0}$ such that, for every $\nu\in\N$,
$$ \left|\de \rho_\nu\right|^2 \;\leq\; \expp{\psi} \;.$$
For every $m\in\N$, set\
$$ \gamma^{(m)} \;:=\; \max_{K^{(m)}} \left(C\cdot \left|\de\psi\right|^2+\expp{\psi}\right) \;, $$
where $C:=:C(n,k)$ is the constant in Proposition \ref{prop:stima}.\\
Fix $\chi\in\mathcal{C}^\infty\left(\R;\R\right)$ such that 
\begin{inparaenum}[(\itshape i\upshape)]
 \item $\chi'>0$,
 \item $\chi''>0$, and
 \item $\chi'\lfloor_{\left(-\infty,\,m\right]}>\frac{\gamma^{(m)}}{L^{(m)}}$, for every $m\in\N$.
\end{inparaenum}
Define
$$ \phi \;:=\; \chi \circ \rho \;: $$
then, $\phi\in \interno{\PSH^{(1)}_p\left( X,\, g\right)} \cap \Cinf{0}$; furthermore
$$ \frac{\del^2\phi}{\del x^{\ell_1}\del x^{\ell_2}} \;=\; \chi''\circ\rho \cdot \frac{\del\rho}{\del x^{\ell_1}} \cdot \frac{\del\rho}{\del x^{\ell_2}}+\chi'\circ\rho \cdot \frac{\del^2\rho}{\del x^{\ell_1}\del x^{\ell_2}} \;.$$
Choose $\mu\in\Cinf{0}$ such that, for every $m\in\N$,
$$ \chi'\circ \rho\lfloor_{K^{(m)}}\cdot L^{(m)} \;\geq\; \mu\lfloor_{K^{(m)}} \;\geq\; \gamma^{(m)} \;.$$

\paragrafo{2}{For every $\eta\in\mathcal{C}^\infty_{\mathrm{c}}\left(X;\wedge^kT^*X\right)$, it holds $\norma{\eta}{\phi-\psi}^2\leq C \cdot \left(\norma{\destar{\phi-\psi}{\phi-2\psi}\eta}{\phi-2\psi}^2 + \norma{\de\eta}{\phi}^2 \right)$}
Since
$$ D^{\left[k\right]}_{g^{-1}\Hess \rho} \;=\; \left(\ssum{|J|=k-1}\sum_{\ell_1,\, \ell_2=1}^{n} \sign{\ell_1 J}{I_1} \sign{\ell_2 J}{I_2} \frac{\del^2\rho}{\del x^{\ell_1}\del x^{\ell_2}}\right)_{I_1, I_2} \;\in\; \Hom{\wedge^{k}TX}{\wedge^{k}TX} \;,$$
one estimates
\begin{eqnarray*}
 \lefteqn{\ssum{\substack{|J|=k-1\\|I_1|=k\\|I_1|=k}}\sum_{\ell_1,\, \ell_2=1}^{n} \sign{\ell_1 J}{I_1} \sign{\ell_2 J}{I_2} \frac{\del^2\phi}{\del x^{\ell_1} \del x^{\ell_2}}\,\eta_{I_1}\,\eta_{I_2}} \\[5pt]
 &=& \ssum{\substack{|J|=k-1\\|I_1|=k\\|I_1|=k}}\sum_{\ell_1,\, \ell_2=1}^{n} \sign{\ell_1 J}{I_1} \sign{\ell_2 J}{I_2} \chi''\circ \rho \cdot \der{\rho}{\ell_1}\der{\rho}{\ell_2}\eta_{I_1}\,\eta_{I_2} \\[5pt]
 && +\; \ssum{\substack{|J|=k-1\\|I_1|=k\\|I_1|=k}}\sum_{\ell_1,\, \ell_2=1}^{n} \sign{\ell_1 J}{I_1} \sign{\ell_2 J}{I_2} \chi'\circ \rho\cdot \frac{\del^2\rho}{\del x^{\ell_1}\del x^{\ell_2}}\, \eta_{I_1}\, \eta_{I_2} \\[5pt]
 &=& \ssum{|J|=k-1}\chi''\circ \rho \cdot \left|\ssum{|I|=k}\sum_{\ell=1}^{n}\sign{\ell J}{I}\der{\rho}{\ell}\eta_I \right|^2 \\[5pt]
 && +\; \chi'\circ\rho \cdot \ssum{\substack{|J|=k-1\\|I_1|=k\\|I_1|=k}}\sum_{\ell_1,\, \ell_2=1}^{n} \sign{\ell_1 J}{I_1} \sign{\ell_2 J}{I_2} \frac{\del^2\rho}{\del x^{\ell_1}\del x^{\ell_2}}\, \eta_{I_1}\, \eta_{I_2} \\[5pt]
 &\geq& \chi'\circ \rho \cdot \lambda_{1}^{[k]}(x) \cdot \ssum{|I|=k}\left|\eta_I\right|^2 \\[5pt]
 &\geq& \mu \cdot \ssum{|I|=k}\left|\eta_I\right|^2 \;.
\end{eqnarray*}
Hence, using Proposition \ref{prop:stima}, we get that, for every $\eta\in\Cinfc{k}$,
\begin{eqnarray*}
 \norma{\eta}{\phi-\psi}^2 &=& \int_X \ssum{|I|=k}\left|\eta_I\right|^2\, \expp{-\left(\phi-\psi\right)}\vol \\[5pt]
 &\leq& \int_X \ssum{|I|=k} \left(\mu-C\cdot\sum_{\ell=1}^{n}\left|\frac{\del\psi}{\del x^\ell}\right|^2\right)\cdot \left|\eta_I\right|^2 \, \expp{-\phi} \vol \\[5pt]
 &\leq& \int_X \left(\ssum{\substack{|J|=k-1\\|I_1|=k\\|I_2|=k}}\sum_{\ell_1,\,\ell_2=1}^{n} \sign{\ell_1 J}{I_1}\sign{\ell_2 J}{I_2} \frac{\del^2\phi}{\del x^{\ell_1}\,\del x^{\ell_2}}\, \eta_{I_1}\, \eta_{I_2} \right. \\[5pt]
 && \left. - C\cdot \ssum{|I|=k}\sum_{\ell=1}^{n}\left|\frac{\del\psi}{\del x^\ell}\right|^2\,\left|\eta_I\right|^2\right)\, \expp{-\phi} \vol \\[5pt]
 &\leq&  C \cdot \left( \norma{\destar{\phi-\psi}{\phi-2\psi}\eta}{\phi-2\psi}^2 + \norma{\de\eta}{\phi}^2 \right) \;, 
\end{eqnarray*}
where $C:=:C(k,n)\in\N$ is the constant in Proposition \ref{prop:stima}, depending just on $k$ and $n$.

\paragrafo{3}{$\mathcal{C}^\infty_{\mathrm{c}}\left(X;\,\wedge^kT^*X\right)$ is dense in $\left(\dom\de\cap\dom\destar{\phi-\psi}{\phi-2\psi},\, \norma{\sspace}{\phi-\psi}+\norma{\destar{\phi-\psi}{\phi-2\psi}\sspace}{\phi-2\psi}+\norma{\de\sspace}{\phi}\right)$}
Consider
$$
\xymatrix{
\Leb{k-1}{\phi-2\psi} \ar@{-->}@/^1pc/[r]^{\de} & \Leb{k}{\phi-\psi} \ar@{-->}@/^1pc/[r]^{\de} \ar@{-->}@/^1pc/[l]^{\destar{\phi-\psi}{\phi-2\psi}} & \Leb{k+1}{\phi} \ar@{-->}@/^1pc/[l]^{\destar{\phi-\psi}{\phi-2\psi}} \;.
}
$$
Fix $\eta\in\dom\de\cap\dom\destar{\phi-\psi}{\phi-2\psi}\subseteq\Leb{k}{\phi-\psi}$. Firstly, we prove that $\left\{\rho_\nu\,\eta\right\}_{\nu\in\N}\subset\dom\de\cap\dom\destar{\phi-\psi}{\phi-2\psi}\subseteq\Leb{k}{\phi-\psi}$ (where $\left\{\rho_\nu\right\}_{\nu\in\N}\subset\Cinfc{0}$ has been defined in {\itshape Step 1}) is a sequence of functions having compact support and converging to $\eta$ in the graph norm $\norma{\sspace}{\phi-\psi}+\norma{\destar{\phi-\psi}{\phi-2\psi}\sspace}{\phi-2\psi}+\norma{\de\sspace}{\phi}$. Indeed,
\begin{eqnarray*}
 \left| \de\left(\rho_\nu\,\eta\right)-\rho_\nu\,\de\eta\right|^2 \,\expp{-\phi} &=& \left|\eta\right|^2\cdot\left|\de\rho_\nu\right|^2 \,\expp{-\phi} \\[5pt]
 &\leq& \left|\eta\right|^2\,\expp{-\left(\phi-\psi\right)} \;\in\; \Leb{k}{0} \;,
\end{eqnarray*}
hence, by Lebesgue's dominated convergence theorem, $\norma{\de\left(\rho_\nu\,\eta\right)-\rho_\nu\,\de\eta}{\phi}\to 0$ as $\nu\to+\infty$. Furthermore, for every $\nu\in\N$, note that $\rho_\nu\,\eta\in\dom\destar{\phi-\psi}{\phi-2\psi}$, since the map
$$ \Leb{k-1}{\phi-2\psi} \;\supseteq\; \dom \de \;\ni\; u \mapsto \scalar{\rho_\nu\,\eta}{\de u}{\phi-\psi} \;\in\; \R $$
is continuous, being
\begin{eqnarray*}
 \scalar{\rho_\nu\,\eta}{\de u}{\phi-\psi} &=& \scalar{\eta}{\de\left(\rho_\nu\, u\right)}{\phi-\psi} - \scalar{\eta}{\de\rho_\nu\wedge u}{\phi-\psi} \\[5pt]
 &=& \scalar{\rho_\nu\,\destar{\phi-\psi}{\phi-2\psi}\eta}{u}{\phi-2\psi} - \scalar{\eta}{\de\rho_\nu\wedge u}{\phi-\psi} \;,
\end{eqnarray*}
hence, by the Riesz representation theorem, there exists $\tilde\eta=:\destar{\phi-\psi}{\phi-2\psi}\left(\rho_\nu\,\eta\right)\in \Leb{k-1}{\phi-2\psi}$ such that, for every $u\in \dom\de \subseteq \Leb{k-1}{\phi-2\psi}$, it holds $\scalar{\rho_\nu\,\eta}{\de u}{\phi-\psi}=\scalar{\tilde\eta}{u}{\phi-2\psi}$. Lastly, note that, for every $u\in\dom\de\subseteq \Leb{k-1}{\phi-2\psi}$,
\begin{eqnarray*}
 \left|\scalar{\destar{\phi-\psi}{\phi-2\psi}\left(\rho_\nu\,\eta\right)-\rho_\nu\,\destar{\phi-\psi}{\phi-2\psi}\,\eta}{u}{\phi-2\psi}\right| &=& \left|\scalar{\rho_\nu\,\eta}{\de u}{\phi-\psi}-\scalar{\destar{\phi-\psi}{\phi-2\psi}\eta}{\rho_\nu\,u}{\phi-2\psi}\right| \\[5pt]
 &=& \left|\scalar{\eta}{\de\rho_\nu\wedge u}{\phi-\psi}\right| \\[5pt]
 &\leq& \norma{\eta}{\phi-\psi}\cdot\norma{\de\rho_\nu\wedge u}{\phi-\psi} \,,
\end{eqnarray*}
hence, by Lebesgue's dominated convergence theorem, $\norma{\destar{\phi-\psi}{\phi-2\psi}\left(\rho_\nu\,\eta\right)-\rho_\nu\,\destar{\phi-\psi}{\phi-2\psi}\eta}{\phi-2\psi}\to 0$ as $\nu\to+\infty$. This shows that $\rho_\nu\,\eta\to \eta$ as $\nu\to+\infty$ with respect to the graph norm.\\
Hence, we may suppose that $\eta\in\dom\de\cap\dom\destar{\phi-\psi}{\phi-2\psi}\subseteq\Leb{k}{\phi-\psi}$ has compact support. Let $\left\{\Phi_\varepsilon\right\}_{\varepsilon\in\R}\subseteq\mathcal{C}^\infty\left(\R^n;\R\right)$ be a family of positive mollifiers, that is, $\Phi_\varepsilon:=\varepsilon^{-n}\,\Phi\left(\frac{\sspace}{\varepsilon}\right)$, where
\begin{inparaenum}[(\itshape i\upshape)]
 \item $\Phi\in\mathcal{C}^\infty_{\textrm{c}}\left(\R^n;\R\right)$,
 \item $\int_{\R^n}\Phi \vol_{\R^n} =1$,
 \item $\lim_{\varepsilon\to 0}\Phi_\varepsilon=\delta$, where $\delta$ is the Dirac delta function, and
 \item $\Phi\geq 0$.
\end{inparaenum}
Consider the convolution $\left\{\eta * \Phi_\varepsilon\right\}_{\varepsilon\in \R}\subset \Cinfc{k}$; we prove that $\eta * \Phi_\varepsilon\to \eta$ as $\varepsilon\to 0$ with respect to the graph norm. Clearly, $\norma{\eta-\eta*\Phi_\varepsilon}{\phi-\psi}\to0$ as $\varepsilon\to 0$. Since $\de\left(\eta*\Phi_\varepsilon\right)=\de\eta*\Phi_\varepsilon$, one has that $\norma{\de\left(\eta*\Phi_\varepsilon\right)-\de\eta}{\phi}\to 0$ as $\varepsilon\to 0$. Lastly, write
$$ \destar{\phi-\psi}{\phi-2\psi} \;=\; \expp{-\psi} \left(\destar{0}{0}+A_{\phi-\psi,\,\phi-2\psi}\right) \;,$$
where $\destar{0}{0}$
is a differential operator with constant coefficients, and $A_{\phi-\psi,\,\phi-2\psi}$ is a differential operator of order zero defined, for every $v\in \Leb{k}{\phi-\psi}$, as
$$ A_{\phi-\psi,\,\phi-2\psi}\left(v\right) \;:=\;  \ssum{\substack{|J|=k-1\\|I|=k}}\sum_{\ell=1}^{n}\sign{\ell J}{I} \der{\left(\phi-\psi\right)}{\ell}\cdot \eta\de x^J \;; $$
hence
\begin{eqnarray*}
\left(\destar{0}{0} +A_{\phi-\psi,\,\phi-2\psi}\right) \left(\eta*\Phi_\varepsilon\right) &=& \left(\left(\destar{0}{0}+A_{\phi-\psi,\,\phi-2\psi}\right) \left(\eta\right)\right) *\Phi_\varepsilon - \left(A_{\phi-\psi,\,\phi-2\psi}\eta\right)*\Phi_\varepsilon + A_{\phi-\psi,\,\phi-2\psi}\left(\eta*\Phi_\varepsilon\right) \\[5pt]
 &\to& \left(\destar{0}{0}+A_{\phi-\psi,\,\phi-2\psi}\right)\left(\eta\right)
\end{eqnarray*}
as $\varepsilon\to0$ in $\Leb{k-1}{\phi-2\psi}$; having $\eta$ compact support, it follows that $\destar{\phi-\psi}{\phi-2\psi}\left(\eta*\Phi_\varepsilon \right)\to \destar{\phi-\psi}{\phi-2\psi}\left(\eta\right)$ as $\varepsilon\to 0$ in $\Leb{k-1}{\phi-2\psi}$.

\paragrafo{4}{If $\norma{\eta}{\phi-\psi}^2\leq C\cdot\left(\norma{\destar{\phi-\psi}{\phi-2\psi}\eta}{\phi-2\psi}^2+\norma{\de\eta}{\phi}^2\right)$ holds for every $\mathcal{C}^\infty_{\mathrm{c}}\left(X;\,\wedge^kT^*X\right)$, then it holds for every $\eta\in\dom\de\cap\dom\destar{\phi-\psi}{\phi-2\psi}$}
Let $\eta\in\dom\de\cap\dom\destar{\phi-\psi}{\phi-2\psi}$. By {\itshape Step 3}, take $\left\{\eta_j\right\}_{j\in\N}\subset \Cinfc{k}$ such that $\eta_j\to \eta$ as $j\to+\infty$ in the graph norm. Since, for every $j\in\N$, one has $\norma{\eta_j}{\phi-\psi}^2\leq C\cdot\left(\norma{\destar{\phi-\psi}{\phi-2\psi}\eta_j}{\phi-2\psi}^2+\norma{\de\eta_j}{\phi}^2\right)$, and since $\norma{\eta_j-\eta}{\phi-\psi}\to 0$, $\norma{\destar{\phi-\psi}{\phi-2\psi}\eta_j-\destar{\phi-\psi}{\phi-2\psi}\eta}{\phi-2\psi}\to 0$ and $\norma{\de\eta_j-\de\eta}{\phi}\to 0$ as $j\to+\infty$, we get that also $\norma{\eta}{\phi-\psi}^2\leq C\cdot\left(\norma{\destar{\phi-\psi}{\phi-2\psi}\eta}{\phi-2\psi}^2+\norma{\de\eta}{\phi}^2\right)$.

\paragrafo{5}{Existence of a solution in $\mathrm{L}^2_{\mathrm{loc}}\left(X;\wedge^{k}T^*X\right)$} We prove here that the operator
$$ \de\colon \Leb{k-1}{\phi-2\psi} \dashrightarrow \ker\left(\de\colon \Leb{k}{\phi-\psi}\dashrightarrow\Leb{k+1}{\phi}\right) $$
is surjective, hence, for every $\eta\in\ker\left(\de\colon \Leb{k}{\phi-\psi}\dashrightarrow\Leb{k+1}{\phi}\right)$, the equation $\de\alpha=\eta$ has a solution $\alpha$ in $\Leb{k-1}{\phi-\psi}\subseteq \Lebloc{k-1}$.\\
We recall (see, e.g., \cite[Lemma 4.1.1]{hormander}) that, given two Hilbert spaces $\left(H_1,\,\scalar{\sspace}{\ssspace}{H_1}\right)$ and $\left(H_2,\,\scalar{\sspace}{\ssspace}{H_2}\right)$, and a densely-defined closed operator $T\colon H_1 \dashrightarrow H_2$, whose adjoint is $T^*\colon H_2\dashrightarrow H_1$, if $F\subseteq H_2$ is a closed subspace such that $\imm T\subseteq F$, then the following conditions are equivalent:
\begin{enumerate}
 \item $\imm T=F$;
 \item there exists $C>0$ such that, for every $y\in \dom T^* \cap F$,
 $$ \norma{y}{H_2} \;\leq\; C\cdot \norma{T^*y}{H_1} \;. $$
\end{enumerate}
Hence, consider
$$ \de\colon \Leb{k-1}{\phi-2\psi} \dashrightarrow \Leb{k}{\phi-\psi} $$
and
\begin{eqnarray*}
\Leb{k}{\phi-\psi} \;\supseteq\; F &:=& \ker\left(\de\colon \Leb{k}{\phi-\psi}\dashrightarrow\Leb{k+1}{\phi}\right) \\[5pt]
  &\supseteq& \imm \left(\de\colon \Leb{k-1}{\psi-2\psi}\dashrightarrow\Leb{k}{\phi-\psi}\right) \;.
\end{eqnarray*}
By {\itshape Step 4}, for every $\eta\in \dom\destar{\phi-\psi}{\phi-2\psi}\cap F \subseteq \dom\de\cap\dom\destar{\phi-\psi}{\phi-2\psi}$, it holds that
$$ \norma{\eta}{\phi-\psi}^2 \;\leq\; C \, \norma{\destar{\phi-\psi}{\phi-2\psi}\eta}{\phi-2\psi}^2 \;,$$
from which it follows that
$$ F \;=\; \imm \left(\de\colon \Leb{k-1}{\psi-2\psi}\dashrightarrow\Leb{k}{\phi-\psi}\right) \;.$$

\paragrafo{6}{Sobolev regularity of the solutions with compact support}

We prove that, for every $\alpha\in\Leb{k-1}{0}$ with compact support, if $\de\alpha\in\Leb{k}{0}$ and $\destar{0}{0}\alpha\in\Leb{k-2}{0}$, then $\alpha\in\Sob{k-1}{0}{1}$. Indeed, take $\left\{\Phi_\varepsilon\right\}_{\varepsilon\in\R}$ a family of positive mollifiers and, for every $\varepsilon\in\R$, consider $\alpha*\Phi_\varepsilon\in\Cinfc{k-1}$; by Remark \ref{rem:stima} with $\phi:=0$ and $\psi:=0$, we get that, for any multi-index $I$ such that $|I|=k-1$ and for any $\ell\in\In{n}$,
$$
\int_X\left|\frac{\del\left(\alpha_I*\Phi_\varepsilon\right)}{\del x^\ell}\right|^2\,\vol \;\leq\; C\cdot\left(\normazero{\destar{0}{0} \left(\alpha*\Phi_\varepsilon\right)}^2+\normazero{\de \left(\alpha*\Phi_\varepsilon\right)}^2\right) \;,
$$
where $C:=:C(k,n)$ is a constant depending just on $k$ and $n$; since, for every multi-index $I$ such that $|I|=k-1$, and for every $\ell\in\In{n}$, it holds that $\lim_{\varepsilon\to0}\int_X\left|\frac{\del\left(\alpha_I*\Phi_\varepsilon\right)}{\del x^\ell} - \frac{\del \alpha_I}{\del x^\ell}\right|^2\,\vol=\lim_{\varepsilon\to0}\normazero{\destar{0}{0} \left(\alpha*\Phi_\varepsilon\right) - \destar{0}{0} \alpha}=\lim_{\varepsilon\to0}\normazero{\de \left(\alpha*\Phi_\varepsilon\right) - \de \alpha}=0$, we get that
$$
\int_X\left|\frac{\del\alpha_I}{\del x^\ell}\right|^2\,\vol \;\leq\; C\cdot\left(\normazero{\destar{0}{0} \alpha}^2+\normazero{\de \alpha}^2\right) \;,
$$
proving the claim.

\paragrafo{7}{Regularization of the solution} By {\itshape Step 5}, if $\eta\in\Cinf{k}$ is such that $\de\eta=0$, then
the equation $\de\alpha=\eta$ has a solution $\alpha\in\Lebloc{k-1}$; we prove that actually $\alpha\in\Cinf{k-1}$.\\
Note that we may suppose that the solution $\alpha\in\Lebloc{k-1}$ satisfies
$$ \alpha \;\in\; \left(\ker\de\right)^{\perp_{\Lebloc{k-1}}} \;=\; \overline{\imm\destar{0}{0}} \;=\; \imm\destar{0}{0} \;\subseteq\; \ker\destar{0}{0} \;;$$
hence, $\alpha$ satisfies the system of differential equation
$$
\left\{
\begin{array}{rcl}
 \de \alpha &=& \eta \\[5pt]
 \destar{0}{0} \alpha &=& 0
\end{array}
\right. \;.
$$
We prove, by induction on $s\in\N$, that $\alpha\in\Sobloc{s}{k-1}$ for every $s\in\N$. Indeed, we have by {\em Step 5} that $\alpha\in\Sobloc{0}{k-1}=\Lebloc{k-1}$. Suppose now that $\alpha\in\Sobloc{s}{k-1}$ and prove that $\alpha\in\Sobloc{s+1}{k-1}$. Clearly, $\eta\in\Cinf{k}\subseteq\Sobloc{\sigma}{k}$ for every $\sigma\in\N$. Take $K$ a compact subset of $X$, and choose $\widehat\chi\in\Cinfc{0}$ such that $\supp\widehat\chi\supset K$. For any multi-index $L:=:\left(\ell_1,\ldots,\ell_n\right)\in\N^n$ such that $\ell_1+\cdots+\ell_n=s$, being
$$ \de\left(\widehat\chi\cdot\frac{\del^{s}\alpha}{\del^{\ell_1} x^1\cdots\del^{\ell_n}x^n}\right) \;=\; \de\widehat\chi\wedge\frac{\del^{s}\alpha}{\del^{\ell_1} x^1\cdots\del^{\ell_n}x^n}+\widehat\chi \cdot \frac{\del^{s}\eta}{\del^{\ell_1} x^1\cdots\del^{\ell_n}x^n} \;\in\; \LebK{k}{0} $$
and
$$ \destar{0}{0}\left(\widehat\chi \cdot \frac{\del^{s}\alpha}{\del^{\ell_1} x^1\cdots\del^{\ell_n}x^n}\right) \;=\;
- \ssum{\substack{|J|=k-1\\|I|=k}}\sum_{\ell=1}^{n}\sign{\ell J}{I}\frac{\del \widehat\chi}{\del x^\ell}\cdot \frac{\del^{s}\alpha_I}{\del^{\ell_1} x^1\cdots\del^{\ell_n}x^n} \de x^J \;\in\; \LebK{k-2}{0} $$
we get that $\widehat\chi\cdot\frac{\del^{s}\alpha}{\del^{\ell_1} x^1\cdots\del^{\ell_n}x^n}\in\SobK{k-1}{0}{1}$, that is, $\alpha\in\SobK{k-1}{0}{s+1}$. Hence, $\alpha\in\Sobloc{s+1}{k-1}$. Since $\Sobloc{\sigma}{k-1} \hookrightarrow \mathcal{C}^m\left(X;\wedge^{k-1}T^*X\right)$ for every $0\leq m< \sigma-\frac{n}{2}$, see \cite[Corollary 7.11]{gilbarg-trudinger}, we get that $\alpha\in\Cinf{k-1}$, concluding the proof of the theorem.
\end{proof}

\end{document}